\documentclass[10pt,twoside]{article}
\usepackage{float,latexsym,graphicx,shortvrb,fancyhdr}
\usepackage{amsmath,amsfonts,bm,mathtools,caption, marvosym,subcaption}
\usepackage[a4paper]{geometry}
\usepackage{enumerate}
\usepackage{mathrsfs}
\usepackage{cite}
\usepackage{graphicx}

\fancypagestyle{firststyle}
{
   
   \fancyhf{}
   \fancyhead[L]{}
   \fancyhead[R]{}
   \fancyfoot[L]{ \footnotesize Z. Jin(\Letter)\\ Department of Mathematics, Shanghai Maritime University, Shanghai, 201306, P.R.China\\ e-mail: asdjz1234@163.com\\
   \vspace{1em}
   D.Y. Gao (\Letter)\\Faculty of Science and Technology, Federation University Australia, Victoria 3353, Australia\\e-mail: d.gao@federation.edu.au
   }
   }

\makeatletter
\renewcommand{\maketitle}{\bgroup\setlength{\parindent}{0pt}
\begin{flushleft}
  \textbf{\@title}

  \@author
\end{flushleft}\egroup
}
\makeatother

\usepackage{amsmath,amsfonts,bm,mathtools,graphicx,latexsym,shortvrb}
\usepackage{xcolor}
\usepackage{hyperref}
\usepackage{multirow}
\usepackage{colortbl}

\newtheorem{theorem}{Theorem}
\newtheorem{lemma}[theorem]{Lemma}
\newtheorem{proposition}[theorem]{Proposition}

\newcommand{\qed}{\hspace*{\fill} $\Box$ \vspace{2ex}}

\def\btheo{\begin{theorem}}
\def\etheo{\end{theorem}}
\def\bprop{\begin{proposition}}
\def\eprop{\end{proposition}}
\def\bexam{\begin{example}}
\def\eexam{\end{example}}
\def\bdefi{\begin{definition}}
\def\edefi{\end{definition}}
\def\blemm{\begin{lemma}}
\def\elemm{\end{lemma}}

\newcommand{\mbb}{\mathbb}
\newcommand{\mcal}{\mathcal}

\newcommand{\sta}{\textrm{sta}}

\def\det{\textrm{det}}

\def\wand{\textrm{ and }}
\def\wotherwise{\textrm{otherwise}}

\def\[#1\]{\begin{align}#1\end{align}}
\def\bcase{\begin{cases}}
\def\ecase{\end{cases}}
\def\bpmat{\begin{pmatrix}}
\def\epmat{\end{pmatrix}}
\def\bbmat{\begin{bmatrix}}
\def\ebmat{\end{bmatrix}}
\def\beqn{\begin{eqnarray}}
\def\eeqn{\end{eqnarray}}
\def\beqnx{\begin{eqnarray*}}
\def\eeqnx{\end{eqnarray*}}
\def\beq{\begin{equation}}
\def\eeq{\end{equation}}
\def\bitem{\begin{itemize}}
\def\eitem{\end{itemize}}
\def\btheo{\begin{theorem}}
\def\etheo{\end{theorem}}
\def\bblock{\begin{block}}
\def\eblock{\end{block}}
\def\benum{\begin{enumerate}}
\def\eenum{\end{enumerate}}

\def\e{\epsilon}

\def\x{\xi}

\def\vP{\varPi}

\def\vX{\varXi}

\def\bbR{\mathbb{R}}

\def\lpa{\left(}
\def\rpa{\right)}

\def\lbc{\left\{}
\def\rbc{\right\}}

\renewcommand{\x}{\chi}

\begin{document}

 \thispagestyle{firststyle}

\title{\bf\Large Canonical dual method for mixed integer fourth-order polynomial minimization problems with  fixed cost terms}
\maketitle

\vspace{3em}

\noindent{{\bf Zhong Jin $\cdot$ David Y Gao}}

\vspace{3em}

\noindent{\textbf{Abstract}}~
we study a canonical duality method to solve a mixed-integer nonconvex fourth-order polynomial minimization problem
with fixed cost terms. This constrained nonconvex problem can be transformed into a continuous concave maximization dual problem without duality gap. The global optimality conditions are proposed and the existence and uniqueness criteria are discussed. Application to a decoupled mixed-integer
problem is illustrated and analytic solution for a global minimum is obtained under some suitable conditions. Several examples are given to show the method is effective.

\vspace{2em}

\noindent{\textbf{Keywords}} ~Canonical duality theory $\cdot$ Mixed-integer programming $\cdot$ Fixed cost objective function $\cdot$ Global optimization

\vspace{2em}

\noindent{\textbf{Mathematics Subject Classification}}~~90C, 49N

\vspace{3em}

\textheight=24cm

%--------------------------------------------------------------------------
\section{Introduction}\label{se:intro}

In this paper, we consider mixed integer minimization problem as follows:

\begin{equation}\label{p:primal}
(\mathcal{P}_b):~~~~\min \{P(x,v)= Q(x)+W(x)-f^Tv~|~(x,v)\in \x_v\}
\end{equation}
where $Q(x)=\frac{1}{2}x^T A x-c^Tx$, $W(x)= \frac{1}{2}\lpa\frac{1}{2}x^T B x-\alpha \rpa^2$, $ A\in\mathbb{R}^{n\times n}$ is a symmetric matrix and $ B\in\mathbb{R}^{n\times n}$ is a symmetric positive semi-definite matrix,  $\alpha>0$ is a real number, $c,f\in \mathbb{R}^n$ are two given vectors, $v\in\{0,1\}^n$ means fixed cost variable, and
\begin{equation}
\x_v=\{(x,v)\in \mathbb{R}^n\times \{0,1\}^n~|~-v\leq x \leq v \}.
\end{equation}

The constraint $x\in [-v,v]$ with $v\in\{0,1\}^n$ is called as fixed charge constraint \cite{Padberg1985}and problem $\mathcal{P}_b$ belongs to the class of fixed charge problems, which have attracted considerable attention and extensive study in the integer programming literature in recent years \cite{Aardal1998,Dukwon1999,Glover2005,Steffen2009,Mathieu2013,Agra2014,Adlakha2015,Housh2015,Gaocost2008}.  This class of problem has
many practical applications in a variety of problems, including facility location\cite{Aardal1998}, network design \cite{Dukwon1999,Steffen2009,Agra2014}, lot-sizing application \cite{Glover2005,Mathieu2013}, and transportation \cite{Mathieu2013,Housh2015,Gaocost2008,Adlakha2015}.The literature provides only a few exact methods for solving fixed charge problems, such as branch-and-bound type
methods, then a good deal of effort has been devoted to
finding approximate solutions to fixed charge problems by heuristic methods. Although these problems can be written as mixed integer programs, their computational requirements usually increase exponentially with
the size of the problem due to the existence of fixed cost terms in its objective function.

Canonical duality theory, as a breakthrough methodological theory, was originally developed from infinite dimensional nonconvex systems \cite{Gaocost2000}. It has been applied successfully for solving a class of challenging problems in global optimization and nonconvex analysis, such
as quadratic problems \cite{FangSC2008,Gao2004,Gao2010}, box constrained problems problems \cite{Gaobox2007a}, polynomial optimization \cite{Zhou2014,Gaobox2007a,Gao2006}, transportation problems
\cite{Gaocost2008}, location problems \cite{Gao2012}, and integer programming problems \cite{FangSC2008,Wang2012}.
 Under some certain conditions, these problems can be solved by canonical
dual method to obtain global minimums of the primal problems.

Moreover, problem $\mathcal{P}_b$ is related with some problems solved by canonical duality method. Particularly, if the fixed cost term $f^Tv=0$ and the feasible space is defined by $\x=\{x\in \mathbb{R}^n~|~\mathscr{L}^l\leq x \leq \mathscr{L}^u\}$ where $\mathscr{L}^l$ and $\mathscr{L}^u$ are two given vectors, then $\mathcal{P}_b$ changes to the nonconvex polynomial programming problem with box constraints in \cite{Gaobox2007a}. If $f^Tv=0$ and the feasible space is simply defined by $\x=\{x\in \mathbb{R}^n~|~x_i =0~or~1, i = 1, \cdots, n\}$, $\mathcal{P}_b$ converts to the simple 0-1 quadratic programming
problem in \cite{FangSC2008}. If $W(x)=0$, $\mathcal{P}_b$ becomes the fixed cost quadratic problem that the canonical duality theory is introduced to solve in \cite{Gaocost2008}, where the existence and uniqueness of global optimal solutions are proved.

In this paper, we design a canonical duality algorithm to solve a mixed-integer nonconvex fourth-order polynomial minimization problem with fixed cost terms. The presented method is inspired by the method introduced by David Y. Gao for fixed cost quadratic programs \cite{Gaocost2008} and provides a nontrivial extension of his work to polynomial minimization problem. Based on the geometrically admissible operators of \cite{Gaobox2007a,Gaocost2008}, we define a new  geometrically admissible operator and transform the fourth-order polynomial minimization problem in $\mathbb{R}^{2n}$ into a continuous concave maximization dual problem in $\mathbb{R}^{n+1}$ without duality gap. We give the global optimality conditions and obtain the existence and uniqueness criteria. Furthermore, we discuss the application to decouple problem. Some numerical experiments are given to show
our method is effective.

This paper is organised as follows. In the next section, the canonical dual problem for $\mathcal{P}_b$ is formulated, and the corresponding complementary-dual principle is showed. The global optimality criteria are put forward in Section 3, and the existence and uniqueness criteria are proposed in Section 4. We apply our method to decoupled problem in Section 5 and illustrate the effectiveness of our method by some examples in Section 6. Some conclusions and further work are given in the last section.

%--------------------------------------------------------------------------
\section{Canonical dual problem}\label{se:cano}
We propose and describe the canonical dual problem of $\mathcal{P}_b$ without
duality gap in this section. First, the box constraints $-v\leq x \leq v$, $v\in \{0,1\}^n$ in the primal problem can be rewritten as relaxed quadratic form:
\begin{equation}\label{constrains}
x\circ x\leq v,~~~~v\circ(v-e)\leq 0,
\end{equation}
where $e=\{1\}^n$ is an n-vector of all ones and notation $x\circ v= (x_1v_1,
x_2v_2,\cdots, x_nv_n)$ denotes the Hadamard product between any two vectors $x,v\in \mathbb{R}^n$. So $\mathcal{P}_b$ can be reformulated to the relaxed problem in the following:
\begin{equation}\label{p:relax}
(\mathcal{P}_r):~~~~\min \{P(x,v)= \frac{1}{2}x^T A x-c^Tx+\frac{1}{2}\lpa\frac{1}{2}x^T B x-\alpha \rpa^2-f^Tv~|~x\circ x\leq v,v\circ(v-e)\leq 0\}.
\end{equation}

In our paper we introduce  a so-called geometrically admissible operator
\begin{equation}\label{operator}
y=\Lambda(x,v)=\bpmat \xi\\ \epsilon \epmat=\bpmat \xi\\ \epsilon_1 \\ \epsilon_2 \epmat=\bpmat
\frac{1}{2}x^T Bx-\alpha\\
x\circ x-v\\
v\circ v-v\\
\epmat
\in\mathbb{R}^{2n+1},
\end{equation}
and it is trivial to see that the constraints (\ref{constrains}) are equivalent to $\epsilon\leq 0$.

Define
\begin{equation}
V(y)=\frac{1}{2}\xi^2+\Psi(\epsilon),
\end{equation}
where
\begin{equation}
     \Psi(\epsilon)=
    \left\{\begin{array}{ll}
    0 &~~ if~~\epsilon\leq 0\in\mathbb{R}^{2n},\\
    +\infty &~~\wotherwise,
    \end{array}\right.
    \end{equation}
and let $U(x,v)=-Q(x)+f^Tv$, then unconstrained canonical form of relaxed primal problem $(\mathcal{P}_r)$
can be obtained:
\begin{equation}\label{pc}
(\mathcal{P}_c):~~~~\min \{\Pi(x,v)= V(\Lambda(x,v))-U(x,v)~|~x\in \mathbb{R}^n,v\in \mathbb{R}^n\}.
\end{equation}

Set $y^*=\left(\begin{array}{l}
\varsigma\\
\sigma
\end{array}\right)=\left(\begin{array}{l}
\varsigma\\
\sigma_1\\
\sigma_2
\end{array}\right)$ be the dual vector of $y=\left(\begin{array}{l}
\xi\\
\epsilon
\end{array}\right)$ with $\epsilon\leq 0$, then sup-Fenchel conjugate of V (y) can be
defined as
\begin{eqnarray}\label{tianjian2002a}
V^\sharp(y^*)&=&\sup_{y\in\bbR^{1+2n}}\{y^Ty^*-V(y)\}\nonumber\\
&=&\sup_{\xi\geq -\alpha}\{\xi\varsigma-\frac{1}{2}\xi^2\}+\sup_{\epsilon_1\in\bbR^{n}}\sup_{\epsilon_2\in\bbR^{n}}\{\epsilon_1^T\sigma_1+\epsilon_2^T\sigma_2-\Psi(\epsilon)\}\nonumber\\
&=&\frac{1}{2}\varsigma^2+\Psi^\sharp(\sigma),~s.t.~\varsigma\geq -\alpha,
\end{eqnarray}
where
\begin{equation}\label{tianjian2002}
     \Psi^\sharp(\sigma)=
    \left\{\begin{array}{ll}
    0 &~~ if~~\sigma_1\geq 0\subseteq\mathbb{R}^{n},\sigma_2\geq 0\subseteq\mathbb{R}^{n}\\
    +\infty &~~\wotherwise.
    \end{array}\right.
    \end{equation}
Now the extended canonical duality
relations can be given from the theory of convex analysis:
\begin{equation}\label{tianjia2004}
y^*\in \partial V(y)~\Leftrightarrow~y\in \partial V^\sharp(y^*)~\Leftrightarrow~V(y)+V^\sharp(y^*)=y^Ty^*,
\end{equation}
and when $\epsilon\leq 0$ and $\sigma\geq 0$ we have
$$\epsilon_1^T\sigma_1+\epsilon_2^T\sigma_2=\frac{1}{2}(\xi-\varsigma)^2\geq 0,$$
then it yields that
\begin{equation}\label{complementarity}
\epsilon_1^T\sigma_1=0~and~\epsilon_2^T\sigma_2=0.
\end{equation}
Apparently it holds from $\epsilon_2^T\sigma_2=\sigma_2^T(v\circ v-v)=0$ in (\ref{complementarity}) that $v\circ v-v=0~for~\forall~\sigma_2>0$.

By replacing $V(\Lambda(x,v))$ in (\ref{pc}) by the Fenchel-Young equality
$V(\Lambda(x,v))=\Lambda(x,v)^Ty^*-V^\sharp(y^*)$, the
total complementary function $\vX(x,v,\varsigma,\sigma)$ corresponding to $(\mathcal{P}_c)$ can be defined in the following:
\begin{eqnarray}\label{duiouhanshu}
\vX(x,v,\varsigma,\sigma)&=&\Lambda(x,v)^Ty^*-V^\sharp(y^*)-U(x,v)\nonumber\\
&=&\frac{1}{2}x^T G(\varsigma,\sigma_1)x-c^Tx+v^T Diag(\sigma_2)v-(f+\sigma_1+\sigma_2)^Tv\nonumber\\
&&-\frac{1}{2}\varsigma^2-\alpha\varsigma-\Psi^\sharp(\sigma),
\end{eqnarray}
where
\begin{eqnarray}\label{tjre001}
G(\varsigma,\sigma_1)=A+\varsigma B+2Diag(\sigma_1),
\end{eqnarray}
and $Diag(\sigma_1)$ denotes a diagonal matrix with diagonal elements ${(\sigma_1)}_i$,
$i = 1, \cdots, n$. Then we get the canonical dual function:
\begin{eqnarray}
\vP^d(\varsigma,\sigma)=\sta\lbc\vX(x,v,\varsigma,\sigma)~|~x\in\mbb{R}^n,v\in\mbb{R}^n\rbc=U^{\Lambda}(\varsigma,\sigma)-V^\sharp(\varsigma,\sigma),
\end{eqnarray}
where $U^{\Lambda}(\varsigma,\sigma)=\sta\lbc\Lambda(x,v)^Ty^*-U(x,v)~|~x\in\mbb{R}^n,v\in\mbb{R}^n\rbc$
is the $\Lambda$-conjugate transformation and the notation $\sta\{\cdot\}$ represents calculating stationary point with respect to $(x,v)$. Correspondingly, introducing a dual feasible space
\begin{equation}\label{inequality}
S_\sharp=\{(\varsigma,\sigma)\in \mathbb{R}^1\times \mathbb{R}^{2n}~|~\varsigma\geq -\alpha,\sigma_1\geq 0,\sigma_2>0, c\in\mathcal{C}_{ol}(G(\varsigma,\sigma_1))\},
\end{equation}
where $\mathcal{C}_{ol}(G)$ denotes the column space of $G$, we can construct the canonical dual function as follows:
\begin{eqnarray}
\vP^d(\varsigma,\sigma)&=&U^{\Lambda}(\varsigma,\sigma)-\frac{1}{2}\varsigma^2\nonumber\\
&=&-\frac{1}{2}c^T G^+(\varsigma,\sigma_1)c-\frac{1}{4}\sum_{i=1}^n\frac{1}{{(\sigma_2)}_i}\lpa f_i+{(\sigma_1)}_i+{(\sigma_2)}_i\rpa^2\nonumber\\
&&-\frac{1}{2}\varsigma^2-\alpha\varsigma,~~\forall (\varsigma,\sigma)\in S_\sharp,\label{2000}
\end{eqnarray}
where $G^+$ stands for the Moore-Penrose generalized inverse of $G$. Then it leads to the canonical dual problem for the primal problem $(\mcal{P}_b)$ as
\begin{equation}\label{p:dual}
(\mcal{P}^\sharp):~~~~\max\lbc\vP^d(\varsigma,\sigma)~|~(\varsigma,\sigma)\in S_\sharp\rbc.
\end{equation}

Next we present the complementary-dual principle and for simplicity we denote $t\oslash s=\{t_i/s_i\}^n$ for any given n-vectors $t=\{t_i\}^n$ and $s=\{s_i\}^n$.

\begin{theorem}\label{th:Complementary-Dual}
If $(\bar\varsigma,\bar\sigma)\in S_\sharp$ is a KKT point of $\vP^d(\varsigma,\sigma)$, then the vector $(\bar x,\bar v)$ is feasible to the primal problem ($\mathcal{P}_b$) and we have
\beq\label{eq:nogap}
P(\bar x, \bar v)=\vX(\bar x,\bar v,\bar \varsigma,\bar \sigma)=\vP^d(\bar \varsigma,\bar \sigma),
\eeq
where
\beq\label{eq:solvedx1}
\bar{x}=\bar{x}(\bar\varsigma,\bar\sigma_1)=G^+(\bar\varsigma,\bar\sigma_1)c,
\eeq
\beq\label{eq:solvedx2}
\bar{v}=\bar{v}(\bar\sigma)=\frac{1}{2}(f+\bar\sigma_1+\bar\sigma_2)\oslash \bar\sigma_2.
\eeq
\end{theorem}
\begin{proof}
Introducing lagrange multiplier $\epsilon=(\epsilon_0,\epsilon_1,\epsilon_2)\in \mathbb{R}^{1}\times\mathbb{R}^{n}\times\mathbb{R}^{n}$ with the respective three inequalities in (\ref{inequality}), we have the lagrangian function $\Theta$ for problem ($\mathcal{P}^\sharp$):
\begin{equation}
\Theta(\varsigma,\sigma,\epsilon_0,\epsilon_1,\epsilon_2)=\vP^d(\varsigma,\sigma)-\epsilon_0(\varsigma+\alpha)
-\epsilon_1^T\sigma_1-\epsilon_2^T\sigma_2.
\end{equation}
It follows from the criticality conditions
$$\nabla_{\sigma_1}\Theta(\bar\varsigma,\bar\sigma,\epsilon_0,\epsilon_1,\epsilon_2)=0,~~~~\nabla_{\sigma_2}\Theta(\bar\varsigma,\bar\sigma,\epsilon_0,\epsilon_1,\epsilon_2)=0$$
that
\begin{equation}\label{2001}
\epsilon_1=\nabla_{\sigma_1}\vP^d(\bar\varsigma,\bar\sigma)=\bar{x}(\bar\varsigma,\bar\sigma_1)\circ \bar{x}(\bar\varsigma,\bar\sigma_1)-\bar{v}(\bar\sigma)
\end{equation}
\begin{equation}\label{2002}
\epsilon_2=\nabla_{\sigma_2}\vP^d(\bar\varsigma,\bar\sigma)=\bar{v}(\bar\sigma)\circ \bar{v}(\bar\sigma)-\bar{v}(\bar\sigma),
\end{equation}
where $\bar{x}(\bar\varsigma,\bar\sigma_1)=G^+(\bar\varsigma,\bar\sigma_1)c$ and $\bar{v}(\bar\sigma)=\frac{1}{2}(f+\bar\sigma_1+\bar\sigma_2)\oslash \bar\sigma_2$.
Then the corresponding KKT conditions include
\begin{equation}\label{2003}
\bar\sigma_1^T\epsilon_1=0~and~\bar\sigma_2^T\epsilon_2=0,
\end{equation}
where $\bar\sigma_1\geq 0$, $\bar\sigma_2>0$, $\epsilon_1\leq 0$ and $\epsilon_2\leq 0$.
By $\epsilon_1\leq 0$ and (\ref{2001}), we have $\bar x\circ \bar x\leq \bar v$. Clearly, together with $\bar\sigma_2>0$  and (\ref{2002}), $\bar\sigma_2^T(\bar v\circ \bar v-\bar v)=0$ in (\ref{2003}) implies that $\bar v\circ \bar v=\bar v$. So when $(\bar\varsigma,\bar\sigma)$ is a KKT point of the problem $\vP^d(\varsigma,\sigma)$, $(\bar x,\bar v)$ is feasible point of ($\mathcal{P}_b$).

Furthermore,
$\bar{v}=\frac{1}{2}(f+\bar \sigma_1+\bar \sigma_2)\oslash \bar \sigma_2$ implies $$\sum_{i=1}^n{(\bar \sigma_2)}_i\lpa\frac{1}{2}\frac{f_i+{(\bar \sigma_1)}_i+{(\bar \sigma_2)}_i}{{(\bar \sigma_2)}_i} \rpa^2=\bar \sigma_2^T(\bar v\circ \bar v)=\bar v^T Diag(\bar \sigma_2)\bar v$$ and $$2 \bar v^T Diag(\bar \sigma_2)\bar v=\bar v^T(f+\bar \sigma_1+\bar \sigma_2),$$
then with $\bar{x}=G^+(\bar \varsigma,\bar \sigma_1)c$ and $\Psi^\sharp(\bar\sigma)=0$, from (\ref{2000}) we get
\begin{eqnarray}\label{2005}
\vP^d(\bar \varsigma,\bar \sigma)
&=&\frac{1}{2}x^T G(\bar \varsigma,\bar \sigma_1)\bar x-c^T \bar x-\bar v^T Diag(\bar \sigma_2)\bar v-\frac{1}{2}\bar \varsigma^2-\alpha\bar \varsigma\nonumber\\
&=&\frac{1}{2}\bar x^T G(\bar \varsigma,\bar \sigma_1)\bar x-c^T \bar x+\bar v^T Diag(\bar \sigma_2)\bar v-\bar v^T(f+\bar\sigma_1+\bar\sigma_2)-\frac{1}{2}\bar \varsigma^2-\alpha\bar \varsigma\\
&=&\vX(\bar x,\bar v,\bar\varsigma,\bar\sigma),\nonumber
\end{eqnarray}
where $(\bar\varsigma,\bar\sigma)\in S_\sharp$. By (\ref{tjre001}), (\ref{2005}),
$\bar v^T Diag(\bar \sigma_2)\bar v=\bar \sigma_2^T(\bar v\circ \bar v)$ and $\bar x^T Diag(\bar \sigma_1)\bar x=\bar \sigma_1^T(\bar x\circ \bar x)$,
we have
\begin{eqnarray}\label{2006}
\vP^d(\bar \varsigma,\bar \sigma)
&=&\vX(\bar x,\bar v,\bar\varsigma,\bar\sigma)\nonumber\\
&=&\frac{1}{2}\bar x^T \Big(A+\bar \varsigma B+2Diag(\bar \sigma_1)\Big)\bar x-c^T \bar x+\bar \sigma_2^T(\bar v\circ \bar v)-\bar v^T(f+\bar \sigma_1+\bar \sigma_2)\nonumber\\
&&-\frac{1}{2}\bar \varsigma^2-\alpha\bar \varsigma\nonumber\\
&=&\frac{1}{2}\bar x^T A\bar x-c^T \bar x-f^T\bar v+\bar \sigma_1^T(\bar x\circ \bar x-\bar v)+\bar \sigma_2^T(\bar v\circ \bar v-\bar v)\nonumber\\
&&+\frac{1}{2}\bar \varsigma \bar x^T B\bar x-\frac{1}{2}\bar \varsigma^2-\alpha\bar \varsigma.
\end{eqnarray}
By substituting $\bar \sigma_1^T(\bar x\circ \bar x-\bar v)=0$ and $\bar\sigma_2^T(\bar v\circ \bar v-\bar v)=0$ into (\ref{2006}), we obtain
\begin{eqnarray}\label{2008}
\vP^d(\bar \varsigma,\bar \sigma)
=\vX(\bar x,\bar v,\bar\varsigma,\bar\sigma)=\frac{1}{2}\bar x^T A\bar x-c^T \bar x-f^T\bar v+\frac{1}{2}\bar \varsigma \bar x^T B\bar x-\frac{1}{2}\bar \varsigma^2-\alpha\bar \varsigma.
\end{eqnarray}

From $\nabla_{\varsigma}\Theta(\bar\varsigma,\bar\sigma,\epsilon_0,\epsilon_1,\epsilon_2)=0$,
together with $\bar{x}=G^+(\bar\varsigma,\bar\sigma_1)c$ it follows that
\begin{equation}\label{tianjia2001}
\frac{1}{2}\bar x^T B \bar x-\bar \varsigma -\alpha-\epsilon_0=0,
\end{equation}
and the accompanying KKT conditions include
\begin{equation}
\bar\varsigma +\alpha\geq 0,~\epsilon_0\leq 0,~\epsilon_0(\bar\varsigma +\alpha)=0.
\end{equation}

Suppose $\epsilon_0< 0$, it holds from $\epsilon_0(\bar\varsigma +\alpha)=0$ that $\bar\varsigma +\alpha=0$, then due to (\ref{tianjia2001}) we have
$$\frac{1}{2}\bar x^T B \bar x-\epsilon_0=0,$$
on the other hand, since $B$ is a symmetric positive semi-definite matrix and $\epsilon_0< 0$, we acquire
$$\frac{1}{2}\bar x^T B \bar x-\epsilon_0>0,$$
which is a contradiction, then $\epsilon_0=0$. Thereby from (\ref{tianjia2001}), we have
\begin{equation}
\frac{1}{2}\bar x^T B \bar x-\bar \varsigma -\alpha=0,
\end{equation}
which implies that $\alpha=\frac{1}{2}\bar x^T B \bar x-\bar \varsigma$
and $\bar \varsigma=\frac{1}{2}\bar x^T B \bar x-\alpha$, then the following equality holds:
$$\frac{1}{2}\bar \varsigma^2+\alpha\bar \varsigma=\frac{1}{2}\bar \varsigma^2+\Big(\frac{1}{2}\bar x^T B \bar x-\bar \varsigma\Big)\bar \varsigma=\frac{1}{2}\bar \varsigma \bar x^T B\bar x-\frac{1}{2}\bar \varsigma^2=\frac{1}{2}\bar \varsigma \bar x^T B\bar x-\frac{1}{2}\lpa\frac{1}{2}\bar x^T B \bar x-\alpha \rpa^2.$$
Thus it is apparent from (\ref{2008}) that
\begin{eqnarray}
\vP^d(\bar \varsigma,\bar \sigma)
=\vX(\bar x,\bar v,\bar\varsigma,\bar\sigma)=\frac{1}{2}\bar x^T A\bar x-c^T \bar x+\frac{1}{2}\lpa\frac{1}{2}\bar x^T B \bar x-\alpha \rpa^2-f^T\bar v
=P(\bar x, \bar v).
\end{eqnarray}
The proof is completed.
\hfill\qed
\end{proof}

\section{Global Optimality Criteria}
The global optimality conditions for
problem $(\mathcal{P}_b)$ are developed in this section. Firstly, we introduce the following useful feasible space:
\begin{equation}\label{3001}
S^+_\sharp=\{(\varsigma,\sigma)\in \mathbb{R}^1\times \mathbb{R}^{2n}~|~\varsigma\geq -\alpha,\sigma_1\geq 0,\sigma_2>0, G(\varsigma,\sigma_1)\succ 0\},
\end{equation}
where $G(\varsigma,\sigma_1)\succ 0$ means that $G(\varsigma,\sigma_1)$ is a positive definite matrix. It is easy to  prove that $S^+_\sharp$ is a convex set. In the following, we use $G$ for short to denote $G(\varsigma,\sigma_1)$.

For convenience, we give the first and second derivatives of function $\vP^d(\varsigma,\sigma)$:
\begin{eqnarray}
\nabla \vP^d(\varsigma,\sigma)
&=&\left(\begin{array}{l}
\frac{1}{2}c^T G^{-1} \frac{\partial G}{\partial\varsigma} G^{-1}c-\varsigma-\alpha\\
\\
\lbc\frac{1}{2}c^T G^{-1} \frac{\partial G}{\partial(\sigma_1)_i} G^{-1}c-\frac{1}{2}\frac{f_i+(\sigma_1)_i+(\sigma_2)_i}{(\sigma_2)_i}\rbc_{i=1}^n\\
\\
\lbc-\frac{1}{2}\frac{f_j+(\sigma_1)_j+(\sigma_2)_j}{(\sigma_2)_j}
+\frac{1}{4}\lpa\frac{f_i+(\sigma_1)_i+(\sigma_2)_i}{(\sigma_2)_i}\rpa^2\rbc_{j=1}^n
\end{array}\right),\label{3004}\\
\nabla^2\vP^d(\varsigma,\sigma)&=&-J_1-J_2-J_3,\label{3005}
\end{eqnarray}
in which $J_1$, $J_2$ and $J_3\in\mbb{R}^{(2n+1)\times (2n+1)}$ are defined as

$$
 J_1=\begin{bmatrix} 1 &  0&  \cdots&  0\\0 &  0& \cdots&  0\\ \vdots&  \vdots& \ddots &  \vdots\\ 0 &  0&  \cdots&  0 \end{bmatrix}, ~~
 J_2=\begin{bmatrix} Z^TG^{-1}Z & \bm 0_{(n+1)\times n}\\\bm 0_{n\times(n+1)} &  \bm 0_{n\times n}  \end{bmatrix}, \wand
J_3=\begin{bmatrix} 0 & \bm 0_{1\times n}& \bm 0_{1\times n}\\\bm 0_{n\times 1}&H_{2\sigma_1^2} & H_{\sigma_1\sigma_2} \\\bm 0_{n\times 1}&H_{\sigma_2\sigma_1} & H_{\sigma_2^2}\end{bmatrix},
$$
where
\begin{eqnarray*}
&& Z=
\begin{bmatrix}
 \frac{\partial G}{\partial\varsigma}G^{-1}c, \frac{\partial G}{\partial(\sigma_1)_1}G^{-1}c, \frac{\partial G}{\partial(\sigma_1)_2}G^{-1}c,\ldots, \frac{\partial G}{\partial(\sigma_1)_n}G^{-1}c
\end{bmatrix},\\
&&H_{2\sigma_1^2}= Diag\lbc-\frac{1}{2(\sigma_2)_i}\rbc,\\
&& H_{\sigma_1\sigma_2} =H_{\sigma_2\sigma_1} = Diag\lbc\frac{(\sigma_1)_i+f_i}{2(\sigma_2)^2_i}\rbc ,
\\
&& H_{\sigma_2^2}= Diag\lbc-\frac{((\sigma_1)_i+f_i)^2}{2(\sigma_2)^3_i}\rbc.
\end{eqnarray*}

\begin{lemma}\label{lemma2}
The canonical dual function $\vP^d(\varsigma,\sigma)$ is concave on $S^+_\sharp$.
\end{lemma}
\begin{proof}
For any given non-zero vector $W=
\left(\begin{array}{l}
r\\
s\\
t
\end{array}\right)\in \mbb{R}^{2n+1}$, where $r\in \mbb{R}^{1}$, $s\in \mbb{R}^{n}$, $t\in \mbb{R}^{n}$,
let $Z_0=\left(\begin{array}{l}
r\\
s
\end{array}\right)$, by (\ref{3005}), with $G\succ 0$ and $\sigma_2>0$ we have
\begin{eqnarray}\label{3006}
W^T\nabla^2\vP^d(\varsigma,\sigma)W&=&-W^TJ_1W-W^TJ_2W-W^TJ_3W\nonumber\\
&=&-r^2-(ZZ_0)^TG^{-1}(ZZ_0)+\sum_{i=1}^n-\frac{1}{2(\sigma_2)_i}\lpa s_i-t_i\frac{(\sigma_1)_i+f_i}{(\sigma_2)_i} \rpa^2\nonumber\\
&\leq&0,
\end{eqnarray}
so the canonical dual function $\vP^d(\varsigma,\sigma)$ is concave on $S^+_\sharp$.
\hfill\qed
\end{proof}

\begin{theorem}\label{theorem4}
Suppose that the vector $\bar{y}^*=(\bar \varsigma,\bar \sigma)=(\bar \varsigma,\bar \sigma_1,\bar \sigma_2)\in S^+_\sharp$ is a critical point
of the dual function $\vP^d(\varsigma,\sigma)$, then $\bar{y}^*$ is a global maximizer of $\vP^d(\varsigma,\sigma)$ on $S^+_\sharp$. Let $(\bar{x}, \bar v)=\lpa G^{-1}(\bar\varsigma,\bar\sigma_1)c,\frac{1}{2}(f+\bar\sigma_1+\bar\sigma_2)\oslash \bar\sigma_2)\rpa$, the $(\bar{x}, \bar v)$ is a global minimum of $P(x,v)$ on $\x_v$(i.e., the $(\bar{x}, \bar v)$ is a global solution of $(\mathcal{P}_b)$), and
\begin{equation}\label{3007}
P(\bar{x}, \bar v)   =\min_{(x,v)\in\x_v} P(x,v)=\max_{(\varsigma,\sigma)\in S^+_\sharp} \vP^d(\varsigma,\sigma)=\vP^d(\bar \varsigma,\bar \sigma).
\end{equation}
\end{theorem}
\begin{proof}
From Lemma \ref{lemma2} the dual function $\vP^d(\varsigma,\sigma)$ is concave on $S^+_\sharp$, so
$\bar{y}^*=(\bar \varsigma,\bar \sigma)$ is a global maximizer of $\vP^d(\varsigma,\sigma)$ on $S^+_\sharp$, i.e.,
\begin{equation}\label{3009}
\vP^d(\bar \varsigma,\bar \sigma)=\max_{(\varsigma,\sigma)\in S^+_\sharp} \vP^d(\varsigma,\sigma),
\end{equation}
and this critical point
of $\vP^d(\varsigma,\sigma)$ is a KKT point of $\vP^d(\varsigma,\sigma)$. Then by Theorem \ref{th:Complementary-Dual},
the vector $(\bar{x}, \bar v)$ defined
by (\ref{eq:solvedx1})(now $G^+(\bar\varsigma,\bar\sigma_1)=G^{-1}(\bar\varsigma,\bar\sigma_1)$) and (\ref{eq:solvedx2}) is a feasible solution to problem $(\mathcal{P}_b)$ and
\begin{equation}\label{3008}
P(\bar x, \bar v)=\vX(\bar x,\bar v,\bar \varsigma,\bar \sigma)=\vP^d(\bar \varsigma,\bar \sigma).
\end{equation}

As $G\succ 0$ and $\sigma_2>0$, we have $\nabla^2_{(x,v)}\vX(x,v,\varsigma,\sigma)\succ 0$
and
$\nabla^2_{(\varsigma,\sigma)}\vX(x,v,\varsigma,\sigma)\preceq 0$. So $\vX(x,v,\varsigma,\sigma)$ is convex in $(x, v)\in \mbb{R}^{2n}=\mbb{R}^{n}\times \mbb{R}^{n}$ and concave in $y^*=(\varsigma,\sigma)\in S^+_\sharp$, then
\begin{equation}\label{3010}
\vP^d(\varsigma,\sigma)=\min_{(x, v)\in \mbb{R}^{2n}} \vX(x,v,\varsigma,\sigma),
\end{equation}
and $(\bar x,\bar v,\bar{y}^*)=(\bar x,\bar v,\bar \varsigma,\bar \sigma)$ is a saddle point of the total complementary function $\vX(\bar x,\bar v,\bar \varsigma,\bar \sigma)$ on $\mbb{R}^{2n}\times S^+_\sharp$, thereby the saddle min-max duality theory holds
\begin{equation}\label{3011}
\max_{(\varsigma,\sigma)\in S^+_\sharp} \min_{(x, v)\in \mbb{R}^{2n}}\vX(x,v,\varsigma,\sigma)= \min_{(x, v)\in \mbb{R}^{2n}}\max_{(\varsigma,\sigma)\in S^+_\sharp}\vX(x,v,\varsigma,\sigma).
\end{equation}
Combining (\ref{3008}), (\ref{3009}), (\ref{3010}) and (\ref{3011}), we get
\begin{equation}\label{3012}
P(\bar x, \bar v)= \min_{(x, v)\in \mbb{R}^{2n}}\max_{(\varsigma,\sigma)\in S^+_\sharp}\vX(x,v,\varsigma,\sigma),
\end{equation}
where $(\bar{x}, \bar v)\in \x_v$.

By (\ref{operator}),
$\Lambda(x,v)=\bpmat \xi\\ \epsilon \epmat=\bpmat \xi\\ \epsilon_1 \\ \epsilon_2 \epmat =\bpmat
\frac{1}{2}x^T Bx-\alpha\\
x\circ x-v\\
v\circ v-v\\
\epmat$.
When $(\varsigma,\sigma)\in S^+_\sharp$, together with (\ref{tianjian2002a}) and (\ref{tianjian2002}), $V^\sharp(y^*)=\frac{1}{2}\varsigma^2$. From (\ref{duiouhanshu}) and $U(x,v)=-Q(x)+f^Tv$, when $(\varsigma,\sigma)\in S^+_\sharp$ we obtain
\begin{eqnarray}\label{3013}
\vX(x,v,\varsigma,\sigma)&=&\Lambda(x,v)^Ty^*-V^\sharp(y^*)-U(x,v)\nonumber\\
&=&\bpmat \xi\\ \epsilon_1 \\ \epsilon_2 \epmat^T\bpmat \varsigma\\ \sigma_1 \\ \sigma_2 \epmat-\frac{1}{2}\varsigma^2-(-Q(x)+f^Tv)\nonumber\\
&=&Q(x)-f^Tv-\frac{1}{2}\varsigma^2+\xi\varsigma+\epsilon_1^T\sigma_1+\epsilon_2^T\sigma_2,
\end{eqnarray}
(a) Consider $\max_{(\varsigma,\sigma)\in S^+_\sharp}(-\frac{1}{2}\varsigma^2+\xi\varsigma)$.
\begin{equation}\label{3014aa}
\max_{(\varsigma,\sigma)\in S^+_\sharp}(-\frac{1}{2}\varsigma^2+\xi\varsigma)=\frac{1}{2}\xi^2= \frac{1}{2}\lpa\frac{1}{2}x^T B x-\alpha \rpa^2=W(x).
\end{equation}
(b) Consider $\max_{(\varsigma,\sigma)\in S^+_\sharp}(\epsilon_1^T\sigma_1+\epsilon_2^T\sigma_2)$.

Noting $\bpmat \epsilon_1 \\ \epsilon_2 \epmat =\bpmat
x\circ x-v\\
v\circ v-v\\
\epmat$, by (\ref{constrains}) we define the relaxed quadratic form region
\begin{equation}\label{3014}
\bar\x_v=\{(x,v)\in \mathbb{R}^n\times \mathbb{R}^n~|~x\circ x\leq v,~v\circ(v-e)\leq 0 \},
\end{equation}
then we have $\x_v\subset\bar\x_v$ and
\begin{equation}\label{3015}
\bpmat \epsilon_1 \\ \epsilon_2 \epmat \leq0 \Longleftrightarrow (x,v)\in \bar\x_v.
\end{equation}
It follows from $(\varsigma,\sigma)\in S^+_\sharp$ that $\sigma_1\geq 0$ and $\sigma_2>0$, thus with (\ref{3015}) we have
\begin{equation}\label{3016}
   \max_{(\varsigma,\sigma)\in S^+_\sharp}(\epsilon_1^T\sigma_1+\epsilon_2^T\sigma_2)=
    \left\{\begin{array}{ll}
    0 &~~ if~~(x,v)\in \bar\x_v,\\
    +\infty &~~\wotherwise(~i.e.,~if~(x,v)\notin \bar\x_v).
    \end{array}\right.
    \end{equation}

Taking (\ref{3013}),  (\ref{3014aa}) and  (\ref{3016}) into consideration,
we find
\begin{eqnarray}\label{3017}
\max_{(\varsigma,\sigma)\in S^+_\sharp}\vX(x,v,\varsigma,\sigma)&=&Q(x)-f^Tv+\max_{(\varsigma,\sigma)\in S^+_\sharp}(-\frac{1}{2}\varsigma^2+\xi\varsigma)+\max_{(\varsigma,\sigma)\in S^+_\sharp}(\epsilon_1^T\sigma_1+\epsilon_2^T\sigma_2)\nonumber\\
&=&
    \left\{\begin{array}{ll}
    Q(x)+W(x)-f^Tv=P(x,v)  &~~ if~~(x,v)\in \bar\x_v,\\
    +\infty &~~\wotherwise.
    \end{array}\right.
\end{eqnarray}
Then it holds form(\ref{3012}) and (\ref{3017}) that
\begin{equation}\label{3018}
P(\bar x, \bar v)= \min_{(x, v)\in \mbb{R}^{2n}}\max_{(\varsigma,\sigma)\in S^+_\sharp}\vX(x,v,\varsigma,\sigma)=\min_{(x, v)\in \bar\x_v}P(x,v).
\end{equation}
Because $\x_v\subset\bar\x_v$, it yields
$$\min_{(x, v)\in \bar\x_v}P(x,v)\leq \min_{(x, v)\in \x_v}P(x,v),$$
which with (\ref{3018}) lead to
$$P(\bar x, \bar v)\leq \min_{(x, v)\in \x_v}P(x,v),$$
which with $(\bar{x}, \bar v)\in \x_v$ implies
\begin{equation}\label{3019}
P(\bar{x}, \bar v)   =\min_{(x,v)\in\x_v} P(x,v).
\end{equation}
Due to the fact that (\ref{3019}), (\ref{3008}) and (\ref{3009}), we get
$$P(\bar{x}, \bar v)   =\min_{(x,v)\in\x_v} P(x,v)=\max_{(\varsigma,\sigma)\in S^+_\sharp} \vP^d(\varsigma,\sigma)=\vP^d(\bar \varsigma,\bar \sigma).$$
\hfill\qed
\end{proof}

Theorem \ref{theorem4} shows that our fourth-order polynomial mixed-integer minimization
problem $(\mathcal{P}_b)$ is canonically dual to the concave maximization
problem as follows:
\begin{equation}\label{3022}
(\mcal{P}^\sharp_+):~~~~\max\lbc\vP^d(\varsigma,\sigma)~|~(\varsigma,\sigma)\in S^+_\sharp\rbc.
\end{equation}
Noted that  $\vP^d(\varsigma,\sigma)$ is a continuous concave function over a convex feasible space
$S^+_\sharp$, if $(\bar\varsigma,\bar\sigma)\in S^+_\sharp$ is a critical point of $\vP^d(\varsigma,\sigma)$, then it must be a global maximum point of problem $(\mcal{P}^\sharp_+)$, and $(\bar{x}, \bar v)=\lpa G^{-1}(\bar\varsigma,\bar\sigma_1)c,\frac{1}{2}(f+\bar\sigma_1+\bar\sigma_2)\oslash \bar\sigma_2)\rpa$ should be
a global minimum point of problem $(\mathcal{P}_b)$.

Using ${(\sigma_2)}_i=|f_i+{(\sigma_1)}_i|$ is the solution of $\min_{{(\sigma_2)}_i>0}\frac{1}{{(\sigma_2)}_i}\lpa f_i+{(\sigma_1)}_i+{(\sigma_2)}_i\rpa^2$,
we have
\begin{equation}\label{3028}
\max_{{(\sigma_2)}_i>0}-\frac{1}{4}\frac{1}{{(\sigma_2)}_i}\lpa f_i+{(\sigma_1)}_i+{(\sigma_2)}_i\rpa^2=-(f_i+{(\sigma_1)}_i)^+.
\end{equation}
So for a fixed $(\varsigma,\sigma_1)$, let
\begin{equation}\label{3025}
\vP^g(\varsigma,\sigma_1)=\max_{\sigma_2>0}\vP^d(\varsigma,\sigma)=-\frac{1}{2}c^T G^{-1}c-\sum_{i=1}^n\lpa f_i+{(\sigma_1)}_i\rpa^+
-\frac{1}{2}\varsigma^2-\alpha\varsigma,~~(\varsigma,\sigma_1)\in S^+_{\varsigma\sigma_1},
\end{equation}
where
\begin{equation}\label{3026}
S^+_{\varsigma\sigma_1}=\{(\varsigma,\sigma_1)\in \mathbb{R}^1\times \mathbb{R}^{n}~|~\varsigma\geq -\alpha,\sigma_1\geq 0, G(\varsigma,\sigma_1)\succ 0, f_i+{(\sigma_1)}_i\neq 0, \forall i=1,\cdots,n\}.
\end{equation}
Then we can write the canonical dual problem $(\mcal{P}^\sharp_+)$ to a simple
form:
\begin{equation}\label{3027}
(\mcal{P}^g_+):~~~~\max\lbc\vP^g(\varsigma,\sigma_1)~|~(\varsigma,\sigma_1)\in S^+_{\varsigma\sigma_1}\rbc.
\end{equation}

Moreover, set $\delta(t)^+=\{\delta_i(t_i)^+\}^n\in \mathbb{R}^n$, where
\begin{equation}\label{3030}
     \delta_i(t_i)^+=
    \left\{\begin{array}{ll}
    1 &~~ if~~t_i>0,\\
    0 &~~if~~t_i<0,
    \end{array}\right. ~~i=1,\cdots,n.
    \end{equation}
By Theorem 3, it can be easily to get next theorem about analytic solution to primal problem $(\mathcal{P}_b)$.
\begin{theorem}\label{theorem7}Given
$ A\in\mathbb{R}^{n\times n}$, $ B\succeq 0\in\mathbb{R}^{n\times n}$, $c,f\in \mathbb{R}^n$,
if $(\bar\varsigma,\bar\sigma_1)\in  S^+_{\varsigma\sigma_1}$ is a critical point of $\vP^g(\varsigma,\sigma_1)$, then the vector
\begin{equation}\label{3031}
(\bar{x}, \bar v)=\lpa G^{-1}(\bar\varsigma,\bar\sigma_1)c,\delta(f+\bar\sigma_1)^+\rpa
\end{equation}
is a global minimum point of $(\mathcal{P}_b)$.
\end{theorem}

\section{Existence and Uniqueness Criteria}
In this section we study certain
existence and uniqueness conditions for the canonical dual problem to have a
critical point. We first let
\begin{equation}\label{4001}
\bar S^+_\sharp=\{(\varsigma,\sigma)\in \mathbb{R}^1\times \mathbb{R}^{2n}~|~\varsigma\geq -\alpha,\sigma_1\geq 0,\sigma_2\geq0, G\succeq 0\},
\end{equation}

\begin{equation}\label{4001.5}
\partial \bar S^+_\sharp=\{(\varsigma,\sigma)\in \bar S^+_\sharp~|~\det (G)=0\}.
\end{equation}

\begin{equation}\label{4002}
\mathfrak{g}_a=\{(\varsigma,\sigma)\in \mathbb{R}^1\times \mathbb{R}^{2n}~|~\sigma_2\geq0, \det (G)=0\}.
\end{equation}

\begin{lemma}\label{lemma7}If $\mathfrak{g}_a\subset \bar S^+_\sharp$, it holds that
$\partial \bar S^+_\sharp=\mathfrak{g}_a$.
\end{lemma}
\begin{proof}
It is obvious that $\bar S^+_\sharp$ is a closed convex set. From $\mathfrak{g}_a\subset \bar S^+_\sharp$, we have
\begin{equation}\label{4005}
\{(\varsigma,\sigma_1)\in \mathbb{R}^1\times \mathbb{R}^{n}~|~\det (G)=0\}\subset\{(\varsigma,\sigma_1)\in \mathbb{R}^1\times \mathbb{R}^{n}~|\varsigma\geq -\alpha,\sigma_1\geq 0, G\succeq 0\}\nonumber
\end{equation}
which shows
\begin{equation}\label{4006}
\{(\varsigma,\sigma_1)\in \mathbb{R}^1\times \mathbb{R}^{n}~|~G\succeq 0\}\subset\{(\varsigma,\sigma_1)\in \mathbb{R}^1\times \mathbb{R}^{n}~|\varsigma\geq -\alpha,\sigma_1\geq 0\}.\nonumber
\end{equation}
Then it holds
\begin{eqnarray}\label{4005}
\bar S^+_\sharp
=\{(\varsigma,\sigma)\in \mathbb{R}^1\times \mathbb{R}^{2n}~|~\sigma_2\geq0, G\succeq 0\},\nonumber
\end{eqnarray}
thus together with (\ref{4001.5}) we have
\begin{equation}\label{4005}
\partial \bar S^+_\sharp=\{(\varsigma,\sigma)\in \mathbb{R}^1\times \mathbb{R}^{2n}~|~\sigma_2\geq0, \det (G)=0\}=\mathfrak{g}_a.
\end{equation}

\hfill\qed
\end{proof}

Motivated by the existence and uniqueness criteria given in \cite{Gao2010,Gaobox2009b}, we have the following theorem about existence and uniqueness criteria.

\begin{theorem}\label{theorem9} Given $A\in\mathbb{R}^{n\times n}$ and a symmetric positive semi-definite matrix $B\in\mathbb{R}^{n\times n}$,  $\alpha>0$, $c,f\in \mathbb{R}^n$, such that $S^+_\sharp\neq\emptyset$ and $\mathfrak{g}_a\subset \bar S^+_\sharp$. If for any given $(\varsigma_0,\sigma_0)\in \mathfrak{g}_a$ and $(\varsigma,\sigma)\in S^+_\sharp$,
\begin{equation}\label{4007}
\lim_{t\rightarrow 0^+}\vP^d(\varsigma_0+t\varsigma,\sigma_0+t\sigma)=-\infty,
\end{equation}
then the canonical dual problem $(\mcal{P}^g_+)$ has at least one critical point $(\bar\varsigma,\bar\sigma_1)\in S^+_{\varsigma\sigma_1}$ and the vector
$(\bar{x}, \bar v)=\lpa G^{-1}(\bar\varsigma,\bar\sigma_1)c,\delta(f+\bar\sigma_1)^+\rpa$
is a global optimizer of the primal problem $(\mathcal{P}_b)$. Furthermore, if $A$, $B$ are two diagonal matrices and
$c_i\neq 0,~\forall i=1,\cdots,n$, then the vector $(\bar{x}, \bar v)$ is a unique global minimum point of $(\mathcal{P}_b)$.
\end{theorem}
\begin{proof}
It holds from $\mathfrak{g}_a\subset \bar S^+_\sharp$ and Lemma \ref{lemma7} that
$\partial \bar S^+_\sharp=\mathfrak{g}_a$. By the definition, $\bar S^+_\sharp$ is a closed convex set and its interior is $S^+_\sharp$. If for any given $(\varsigma,\sigma)\in S^+_\sharp$
\begin{eqnarray}\label{4008}
\lim_{t\rightarrow +\infty}\vP^d(t\varsigma,t\sigma)
&=&\lim_{t\rightarrow +\infty}(-\frac{1}{2}c^T G^{-1}(t\varsigma,t\sigma_1)c-\frac{1}{4}\sum_{i=1}^n\frac{1}{t{(\sigma_2)}_i}\lpa f_i+t{(\sigma_1)}_i+t{(\sigma_2)}_i\rpa^2\nonumber\\
&&-\frac{1}{2}t^2\varsigma^2-\alpha t\varsigma)=-\infty,
\end{eqnarray}
which with (\ref{4007}) implies that the function $\vP^d(\varsigma,\sigma)$ is coercive on the open convex
set $S^+_\sharp$. Therefore, $(\mcal{P}^\sharp_+)$ has at least one critical point $(\bar\varsigma,\bar\sigma)\in S^+_\sharp$ and the vector $(\bar{x}, \bar v)=(G^{-1}(\bar\varsigma,\bar\sigma_1)c,\frac{1}{2}(f+\bar\sigma_1+\bar\sigma_2)\oslash \bar\sigma_2))$ is
a global minimum point of problem $(\mathcal{P}_b)$. So accordingly $(\mcal{P}^g_+)$ has at least one critical point $(\bar\varsigma,\bar\sigma_1)\in  S^+_{\varsigma\sigma_1}$ and the vector $(\bar{x}, \bar v)=\lpa G^{-1}(\bar\varsigma,\bar\sigma_1)c,\delta(f+\bar\sigma_1)^+\rpa$
is a global minimum point of $(\mathcal{P}_b)$.

Next we will prove the canonical dual function $\vP^d(\varsigma,\sigma)$ is strict concave on $S^+_\sharp$.

From (\ref{3006}), we have
\begin{equation}\label{4009}
W^T\nabla^2\vP^d(\varsigma,\sigma)W=-r^2-(ZZ_0)^TG^{-1}(ZZ_0)+\sum_{i=1}^n-\frac{1}{2(\sigma_2)_i}\lpa s_i-t_i\frac{(\sigma_1)_i+f_i}{(\sigma_2)_i} \rpa^2.
\end{equation}
(a) If $Z_0=\left(\begin{array}{l}
r\\
s
\end{array}\right)=0$, $W\neq 0$ leads to $t\neq 0$, then with ${(\sigma_2)}_i>0$ and $f_i+{(\sigma_1)}_i\neq 0$, we get $$W^T\nabla^2\vP^d(\varsigma,\sigma)W=\sum_{i=1}^n-\frac{1}{2(\sigma_2)_i}\lpa 0-t_i\frac{(\sigma_1)_i+f_i}{(\sigma_2)_i} \rpa^2=\sum_{i=1}^n-\frac{t_i^2}{2((\sigma_2)_i)^3}\lpa f_i+(\sigma_1)_i \rpa^2<0.$$
(b) If $Z_0=\left(\begin{array}{l}
r\\
s
\end{array}\right)\neq 0$, there are two cases.

(b.1) If $r\neq 0$, then $W^T\nabla^2\vP^d(\varsigma,\sigma)W\leq -r^2<0$.

(b.2) If $r=0$, then $s\neq 0$ and $Z_0=\left(\begin{array}{l}
0\\
s
\end{array}\right)$, thus
\begin{equation}\label{4010}
-(ZZ_0)^TG^{-1}(ZZ_0)=-4s^TDiag(x(\varsigma,\sigma_1))^TG^{-1}Diag(x(\varsigma,\sigma_1))s.
\end{equation}
Since $A$ and $B$ are two diagonal matrices, then from (\ref{tjre001}) we know $G$ is a diagonal matrix, which with $G\succ 0$ shows that $G^{-1}$ is a diagonal matrix whose all diagonal elements are positive. Then it follows from $x(\varsigma,\sigma_1) = G^{-1}c$ and $c_i\neq 0,~\forall i=1,\cdots,n$ that
$$x(\varsigma,\sigma_1)_i\neq 0,~\forall i=1,\cdots,n.$$
By $G^{-1}\succ 0$, we find $Diag(x(\varsigma,\sigma_1))^TG^{-1}Diag(x(\varsigma,\sigma_1))\succ 0$, then it is observed from (\ref{4010}) that
$$-(ZZ_0)^TG^{-1}(ZZ_0)<0,$$
hence
$$W^T\nabla^2\vP^d(\varsigma,\sigma)W\leq -(ZZ_0)^TG^{-1}(ZZ_0)<0.$$

From above two cases, we can see $\vP^d(\varsigma,\sigma)$ is strictly concave on $S^+_\sharp$.
Thereby $(\mcal{P}^\sharp_+)$ has a unique critical point in $S^+_\sharp$, which implies $(\mcal{P}^g_+)$ has
a unique critical point in $S^+_{\varsigma\sigma_1}$ and the primal problem has a unique global
minimum.

\hfill\qed
\end{proof}

\section{Application to Decoupled Problem}
In this section, we discuss the application of presented theory to the decoupled problems. Consider the decoupled minimization problem as follows:
\begin{equation}\label{5001}
\min \left\{P(x,v)= \frac{1}{2}\Big(\frac{1}{2}\sum_{i=1}^nb_ix_i^2-\alpha \Big)^2+\sum_{i=1}^n\lpa\frac{1}{2}a_ix_i^2-c_ix_i-f_iv_i\rpa\right\}
\end{equation}
\begin{equation}\label{5002}
s.t.~~-v_i\leq x_i\leq v_i,v_i\in\{0,1\}, i=1,\ldots,n.
\end{equation}
For simplicity, we may take $A = Diag (a)$
and $B = Diag (b)$ stand for two diagonal matrices with diagonal elements $a = \{a_i\} \in \mathbb{R}^{n}$ and $b = \{b_i\} \in \mathbb{R}^{n}$
respectively. The canonical dual function of above problem has the following simple form
\begin{equation}\label{5003}
\vP^d(\varsigma,\sigma)=-\frac{1}{2}\sum_{i=1}^n\lpa\frac{c_i^2}{a_i+\varsigma b_i+2{(\sigma_1)}_i}+\frac{\lpa f_i+{(\sigma_1)}_i+{(\sigma_2)}_i\rpa^2}{2{(\sigma_2)}_i}\rpa-\frac{1}{2}\varsigma^2-\alpha\varsigma.
\end{equation}
Due to $\nabla\vP^d(\varsigma,\sigma)=0$, then when ${(\sigma_2)}_i\neq0$ and $c_i\neq 0$, we can obtain the critical points of $\vP^d(\varsigma,\sigma)$ in the following
\begin{equation}\label{50041}
\varsigma=\frac{1}{2}\sum_{i=1}^nb_i-\alpha,~~{(\sigma_1)}_i\in M_i=\left\{-\frac{1}{2}\Big( a_i+\Big(\frac{1}{2}\sum_{i=1}^nb_i-\alpha\Big)b_i\pm c_i\Big)\right\},
\end{equation}
\begin{equation}\label{50042}
{(\sigma_2)}_i\in N_i=\left\{f_i-\frac{1}{2}\Big( a_i+\Big(\frac{1}{2}\sum_{i=1}^nb_i-\alpha\Big)b_i\pm c_i\Big)\right\},~~\forall i=1,\ldots,n.
\end{equation}
For ${(\sigma_2)}_i>0$, using Theorem \ref{th:Complementary-Dual}, we can derive the accompanying primal solution
\begin{equation}\label{5005}
(x_i,v_i)=\lpa-\frac{c_i}{a_i+(\frac{1}{2}\sum_{i=1}^nb_i-\alpha) b_i+2{(\sigma_1)}_i},\frac{ f_i+{(\sigma_1)}_i+{(\sigma_2)}_i}{2{(\sigma_2)}_i}\rpa,~~\forall i=1,\ldots,n.
\end{equation}
According to the fact that there are two possible solutions for each component of $\sigma=(\sigma_1,\sigma_2)\in \mathbb{R}^{2n}$,
together with (\ref{50041}) and (\ref{50042}), it follows that the canonical dual function $\vP^d$ has $2^n$ critical points.
Then from Theorem \ref{theorem4}, it is easy to show the global minimum of the primal problem will be arrived at by next theorem.
\begin{theorem}\label{theorem10}
Given $A=Diag(a)\in\mathbb{R}^{n\times n}$ and a symmetric positive semi-definite matrix $B=Diag(b)\in\mathbb{R}^{n\times n}$,  $\alpha>0$, $c,f\in \mathbb{R}^n$,
$c_i\neq 0$ for $\forall i$, if
\begin{equation}
\max M_i>0~~and~~\max N_i>0,~~\forall i=1,\ldots,n.
\end{equation}
then $\vP^d(\varsigma,\sigma)$ has a unique critical point
\begin{eqnarray}
(\bar \varsigma,\bar \sigma)=(\bar \varsigma,\bar \sigma_1,\bar \sigma_2)
=\Bigg(\frac{1}{2}\sum_{i=1}^nb_i-\alpha,\big\{\max M_i,i=1,\ldots,n\big\},\big\{\max N_i,i=1,\ldots,n\big\}\Bigg)\in S^+_\sharp,\nonumber
\end{eqnarray}
which is a global maximizer of $\vP^d(\varsigma,\sigma)$ on $S^+_\sharp$, and
$$(\bar x,\bar v)=\lpa\left\{\frac{c_i}{|c_i|} \right\}, e \rpa$$
is a global minimum of $P(x, v)$ on $\x_v$.
\end{theorem}

\section{Examples}
Now we give a summary of numerical experiments to illustrate our method, where diagonal matrices $A$ and $B$, vectors $f$ and $c$ are chosen and taken at random. These examples are classified into three cases and in every case we give several representatives. In the first case we consider the decoupled problems satisfying the conditions of Theorem \ref{theorem10}, whose results are consistent with Theorem \ref{theorem10} and show our method is promising for decoupled problems. In the second case some general decoupled problems not satisfying Theorem \ref{theorem10} are computed by our method and the global solution are also obtained.  In the last case, our method is tested for some general problems. All of performed examples show our method is very effective.

\subsection{Case 1: Decoupled Problems satisfying Theorem \ref{theorem10}}
In the following three decoupled examples, we can verify the conditions of Theorem \ref{theorem10} are satisfied, so
 a unique critical point $(\bar\varsigma,\bar\sigma)$ of $\vP^d(\varsigma,\sigma)$ on $S^+_\sharp$ and a global minimum point $(\bar x,\bar v)$ of $P(x, v)$ on $\x_v$ are obtained by Theorem \ref{theorem10}. For simplicity we denote $\lambda_{min}(\bar\varsigma,\bar\sigma_1)$ be the smallest eigenvalue of $G$.
 \newtheorem{example}{{Example}}
\begin{example}
Set $\alpha=10$, $A=Diag (1,-1,1,5,2)$, $B=Diag (2,4,1,4,2)$, $f=(20,12,-1,$ $1,13)$ and $c=(-8,-9,10,9,-5)$, in which $n=5$.
\end{example}
\begin{example}
Set $\alpha=20$, $A=Diag (7,9,6,-5,4,10,9,8)$, $B=Diag (3,5,4,3,1,7,5,7)$, $f=(13,-3,3,11,10,16,16,14)$ and $c=(7,-6,8,-1,-5,8,-8,7)$, in which $n=8$.
\end{example}
\begin{example}
Set $\alpha=25$, $A=Diag (2,8,7,3,6,14,10,1,-6,9)$, $B=Diag (9,1,2,1,6,8,5,$ $3,9,6)$, $f=(6,1,4,13,6,15,17,20,3,16)$ and $c=(19,14,-9,-9,-8,17,-22,-14,-8,18)$, in which $n=10$.
\end{example}

The values of $\frac{1}{2}\sum_{i=1}^nb_i-\alpha$, $M_i$ and $N_i$ from (\ref{50041}) and (\ref{50042}) are first needed to be computed and are listed
in Table 1 as follows. Then it can be found that conditions of Theorem \ref{theorem10} are all satisfied.

\begin{table}[H]
\centering
\scalebox{0.8}[0.8]{%
\begin{tabular}{|c|c|l|l|l|l|}
\hline
Experiments &$\frac{1}{2}\sum_{i=1}^nb_i-\alpha$ &$M_i$ & $Max~M_i$ & $N_i$& $Max~N_i$\\
\hline
\multirow{5}{*}{Example 1}&\multirow{5}{*}{-3.5} &$M_1=\{7,-1\}$ & 7 & $N_1=\{27,19\}$ & 27\\
&& $M_2=\{12,3\}$ & 12 & $N_2=\{24,15\}$& 24\\
& &$M_3=\{-3.75,6.25\}$ & 6.25 &$N_3=\{-4.75,5.25\}$ & 5.25\\
& &$M_4=\{0,9\}$ & 9 &$N_4=\{1,10\}$ & 10\\
& &$M_5=\{5,0\}$ & 5 & $N_5=\{18,13\}$& 18\\
\hline
\multirow{8}{*}{Example 2}&\multirow{8}{*}{-2.5} &$M_1=\{-3.25,3.75\}$ & 3.75& $N_1=\{9.75,16.75\}$ & 16.75 \\
&& $M_2=\{4.75,-1.25\}$ & 4.75 & $N_2=\{1.75,-4.25\}$& 1.75\\
& &$M_3=\{-2,6\}$ & 6 &$N_3=\{1,9\}$ & 9\\
& &$M_4=\{6.75,5.75\}$ & 6.75 &$N_4=\{17.75,16.75\}$ & 17.75\\
& &$M_5=\{1.75,-3.25\}$ & 1.75 & $N_5=\{11.75,6.75\}$& 11.75\\
& &$M_6=\{-0.25,7.75\}$ & 7.75 &$N_6=\{15.75,23.75\}$ & 23.75\\
& &$M_7=\{5.75,-2.25\}$ & 5.75 &$N_7=\{21.75,13.75\}$ & 21.75\\
& &$M_8=\{1.25,8.25\}$ & 8.25 & $N_8=\{15.25,22.25\}$& 22.25\\
\hline
\multirow{10}{*}{Example 3}&\multirow{10}{*}{0} &$M_1=\{-10.5,8.5\}$ & 8.5 &$N_1=\{-4.5,14.5\}$ & 14.5\\
&& $M_2=\{-11,3\}$ & 3 & $N_2=\{-10,4\}$& 4\\
& &$M_3=\{1,-8\}$ & 1 &$N_3=\{5,-4\}$ & 5\\
& &$M_4=\{3,-6\}$ & 3 &$N_4=\{16,7\}$ & 16\\
& &$M_5=\{1,-7\}$ & 1 & $N_5=\{7,-1\}$& 7\\
& &$M_6=\{-15.5,1.5\}$ & 1.5 &$N_6=\{-0.5,16.5\}$ & 16.5\\
& &$M_7=\{6,-16\}$ & 6 &$N_7=\{23,1\}$ & 23\\
& &$M_8=\{6.5,-7.5\}$ & 6.5 & $N_8=\{26.5,12.5\}$& 26.5\\
& &$M_9=\{7,-1\}$ & 7 &$N_9=\{10,2\}$ & 10\\
& &$M_{10}=\{-13.5,4.5\}$ & 4.5 & $N_{10}=\{2.5,20.5\}$& 20.5\\
\hline
\end{tabular}}
\caption{Conditions of Theorem \ref{theorem10} are satisfied in Examples 1-3.}
\end{table}

On the one hand, by Theorem \ref{theorem10} we know $\varsigma=\frac{1}{2}\sum_{i=1}^nb_i-\alpha$, $(\bar \sigma_1)_i=\max M_i$, $(\bar \sigma_2)_i=\max N_i$ and $(\bar x,\bar v)=\lpa\left\{\frac{c_i}{|c_i|} \right\}, e \rpa$,
so from Table 1 we can easily get the corresponding $(\bar\varsigma,\bar\sigma)$ and $(\bar x,\bar v)$ for Examples 1-3 listed in Table 2.
\begin{table}[!hbp]
\centering
\scalebox{0.78}[0.78]{%
\begin{tabular}{|c|l|l|l|}
\hline
Experiments &$\bar\varsigma$ &$\bar\sigma$  & Primal solution $(\bar x, \bar v)$ \\
\hline
\multirow{2}{*}{Example 1}&$\bar\varsigma=-3.5$ &$\bar\sigma_1=(7,12,6.25,9,5)$, &$\bar x=(-1,-1,1,1,-1)$, \\
& & $\bar\sigma_2=(27,24,5.25,10,18)$.&  $\bar v=(1,1,1,1,1)$.  \\
\hline
\multirow{3}{*}{Example 2} & $\bar\varsigma=-2.5$ &$\bar\sigma_1=(3.75,4.75,6,6.75,1.75,7.75,5.75,8.25)$,  &$\bar x=(1,-1,1,-1,-1,1,-1,1)$\\
&   & $\bar\sigma_2=(16.75,1.75,9,17.75,11.75,23.75,21.75,22.25)$.&$\bar v=(1,1,1,1,1,1,1,1)$.  \\
\hline
\multirow{3}{*}{Example 3} &$\bar\varsigma=0$  &$\bar\sigma_1=(8.5,3,1,3,1,1.5,6,6.5,7,4.5)$, &$\bar x=(1,1,-1,-1,-1,1,-1,-1,-1,1)$, \\
&   & $\bar\sigma_2=(14.5,4,5,16,7,16.5,23,26.5,10,20.5)$.& $\bar v=(1,1,1,1,1,1,1,1,1,1)$. \\
\hline
\end{tabular}}
\caption{Results of Theorem \ref{theorem10} in Examples 1-3.}
\end{table}

On the other hand, we can also get
$(\bar\varsigma,\bar\sigma)$ by solving the canonical dual problem $(\mcal{P}^\sharp_+)$, then $(\bar x,\bar v)$ can be computed by $(\bar{x}, \bar v)=\lpa G^{-1}(\bar\varsigma,\bar\sigma_1)c,\frac{1}{2}(f+\bar\sigma_1+\bar\sigma_2)\oslash \bar\sigma_2)\rpa$. The corresponding results are listed below:

\begin{table}[!hbp]
\centering
\scalebox{0.7}[0.7]{%
\begin{tabular}{|c|l|c|c|l|c|}
\hline
Experiments &Dual solution $(\bar\varsigma,\bar\sigma)$ &$\vP^d(\bar\varsigma,\bar\sigma)$ &$\lambda_{min}(\bar\varsigma,\bar\sigma_1)$ & Primal solution $(\bar x, \bar v)$& $P(\bar x, \bar v)$ \\
\hline
\multirow{3}{*}{Example 1}&$\bar\varsigma=-3.5$, &\multirow{3}{*}{-75.875} & \multirow{3}{*}{5}&$\bar x=(-1,-1,1,1,-1)$,& \multirow{3}{*}{-75.875} \\
& $\bar\sigma_1=(7,12,6.25,9,5)$,& & & $\bar v=(1,1,1,1,1)$. & \\
& $\bar\sigma_2=(27,24,5.25,10,18)$.& & & & \\
\hline
\multirow{3}{*}{Example 2} & $\bar\varsigma=-2.5$, &\multirow{3}{*}{-102.875} & \multirow{3}{*}{1}&$\bar x=(1,-1,1,-1,-1,1,-1,$& \multirow{3}{*}{-102.875} \\
& $\bar\sigma_1=(3.75,4.75,6,6.75,1.75,7.75,5.75,8.25)$,  & && $1)$, & \\
& $\bar\sigma_2=(16.75,1.75,9,17.75,11.75,23.75,21.75,22.25)$. &  &&$\bar v=(1,1,1,1,1,1,1,1)$.  & \\
\hline
\multirow{3}{*}{Example 3} &$\bar\varsigma=0$,  &\multirow{3}{*}{-212}&\multirow{3}{*}{8} &$\bar x=(1,1,-1,-1,-1,1,-1,$ & \multirow{3}{*}{-212} \\
& $\bar\sigma_1=(8.5,3,1,3,1,1.5,6,6.5,7,4.5)$,  & & &$-1,-1,1)$, & \\
& $\bar\sigma_2=(14.5,4,5,16,7,16.5,23,26.5,10,20.5)$. &  & &$\bar v=(1,1,1,1,1,1,1,1,1,1)$. & \\
\hline
\end{tabular}}
\caption{Results by solving problem $(\mcal{P}^\sharp_+)$ in Examples 1-3.}
\end{table}

From Tables 2 and 3, it can be seen that the results from Theorem \ref{theorem10} are consistent with the ones by our canonical dual method. And
 the fact that $\bar\varsigma\geq -\alpha,\bar\sigma_1\geq 0,\bar\sigma_2\geq0, G(\bar\varsigma,\bar\sigma_1)\succeq 0$ in every example indicates
$(\bar \varsigma,\bar \sigma)\in S^+_\sharp$. By Theorem \ref{theorem10}, the solution of primal problem is obtained. It is verified that our method is promising for decoupled problems when the conditions of Theorem \ref{theorem10} are satisfied for the decoupled problems.

\subsection{Part 2: Decoupled Problems where Theorem \ref{theorem10} are not satisfied}
It needs to say that there are many decoupled problems not to satisfy Theorem \ref{theorem10}. Here we choose three of them to indicate the details. Let $n=5$ and $\alpha=10$, diagonal matrices $A$ and $B$, vectors $f$ and $c$ are chosen and taken at random.

\begin{example}
Set $A=Diag (6,3,9,9,2)$, $B=Diag (2,4,5,4,3)$, $f=(5,4,4,20,9)$ and $c=(1,-9,-6,3,-5)$.
\end{example}
\begin{example}
Set $A=Diag (1,-1,1,4,4)$, $B=Diag (1,1,1,4,5)$, $f=(1,-51,-1,-11,-61)$ and $c=(3,0,1,-2,0)$.
\end{example}
\begin{example}
Set $A=Diag (5,-1,2,5,1)$, $B=Diag (5,2,2,1,4)$, $f=(3,-35,-1,11,15)$ and $c=(7,0,4,-6,10)$.
\end{example}

\begin{table}[H]
\centering
\scalebox{0.8}[0.8]{%
\begin{tabular}{|c|c|l|l|l|l|}
\hline
Experiments &$\frac{1}{2}\sum_{i=1}^nb_i-\alpha$ &$M_i$ & $Max~M_i$ & $N_i$& $Max~N_i$\\
\hline
\multirow{5}{*}{Example 4}&\multirow{5}{*}{-1} &$M_1=\{-2.5,-1.5\}$ & -1.5 & $N_1=\{2.5,3.5\}$ & 3.5\\
&& $M_2=\{5,-4\}$ & 5 & $N_2=\{9,0\}$& 9\\
& &$M_3=\{1,-5\}$ & 1 &$N_3=\{5,-1\}$ & 5\\
& &$M_4=\{-4,-1\}$ & -1 &$N_4=\{16,19\}$ & 19\\
& &$M_5=\{3,-2\}$ & 3 & $N_5=\{12,7\}$& 12\\
\hline
\end{tabular}}
\caption{Conditions of Theorem \ref{theorem10} are not all satisfied in Examples 4.}
\end{table}

From Table 4, we find that $Max~M_1=-1.5<0$ which makes the conditions of Theorem \ref{theorem10} are not satisfied. We solve the simple form problem $(\mcal{P}^g_+)$ instead of the canonical dual problem $(\mcal{P}^\sharp_+)$ and use
Theorem \ref{theorem7} to obtain the analytic solution $(x, v)$ to primal problem $(\mathcal{P}_b)$. The corresponding results are listed below:

\begin{table}[!hbp]
\centering
\scalebox{0.8}[0.8]{%
\begin{tabular}{|c|l|c|c|l|c|}
\hline
Experiments &Dual solution $(\bar\varsigma,\bar\sigma_1)$ &$\vP^g(\bar\varsigma,\bar\sigma_1)$ &$\lambda_{min}(\bar\varsigma,\bar\sigma_1)$ & Primal solution $(\bar x, \bar v)$& $P(\bar x, \bar v)$ \\
\hline
\multirow{2}{*}{Example 4}&$\bar\varsigma=-1.82$, &\multirow{2}{*}{-51.7281} & \multirow{2}{*}{2.3593}& $\bar x=(0.424,-1,-1,1,-1)$,& \multirow{2}{*}{-51.7281} \\
& $\bar\sigma_1=(0,6.641,3.051,0.641,4.231)$.& & & $\bar v=(1,1,1,1,1)$. & \\
\hline
\end{tabular}}
\caption{Results by solving problem $(\mcal{P}^g_+)$ in Examples 4.}
\end{table}

It is obvious that there exists some $c_i=0$ in Examples 5 and 6, which does not satisfy the conditions of Theorem \ref{theorem10}. We also solve them by the simple form problem $(\mcal{P}^g_+)$, whose results are illustrated in Table 6.

\begin{table}[!hbp]
\centering
\scalebox{0.85}[0.85]{%
\begin{tabular}{|c|l|c|c|l|c|}
\hline
Experiments &Dual solution $(\bar\varsigma,\bar\sigma_1)$ &$\vP^g(\bar\varsigma,\bar\sigma_1)$ &$\lambda_{min}(\bar\varsigma,\bar\sigma_1)$& Primal solution $(\bar x, \bar v)$& $P(\bar x, \bar v)$ \\
\hline
\multirow{2}{*}{Example 5}&$\bar\varsigma=-7$, &\multirow{2}{*}{32.5} &\multirow{2}{*}{1} &$\bar x=(1,0,1,-1,0)$,& \multirow{2}{*}{32.5} \\
& $\bar\sigma_1=(4.5,34.987,3.5,13,54.367)$.& & & $\bar v=(1,0,1,1,0)$. & \\
\hline
\multirow{2}{*}{Example 6} & $\bar\varsigma=-4$, &\multirow{2}{*}{-40.5 } & \multirow{2}{*}{2.22424 }&$\bar x=(1,0,1,-1,1)$,& \multirow{2}{*}{-40.5 } \\
& $\bar\sigma_1=(11,5.612,5,2.5,12.5)$.  &  && $\bar v=(1,0,1,1,1)$. & \\
\hline
\end{tabular}}
\caption{Results by solving problem $(\mcal{P}^g_+)$ in Examples 5 and 6.}
\end{table}

From Tables 5-6, we have $\bar\varsigma\geq -\alpha,\bar\sigma_1\geq 0, G(\bar\varsigma,\bar\sigma_1)\succ 0, f_i+{(\bar\sigma_1)}_i\neq 0, \forall i=1,\cdots,n$ in every example, so
$(\bar \varsigma,\bar \sigma_1)\in S^+_{\varsigma\sigma_1}$. By Theorem \ref{theorem7}, the solution of primal problem is obtained. Thus when the conditions of Theorem \ref{theorem10} are not all satisfied, our method is also effective for the decoupled problem.

\subsection{Part 3: General Nonconvex Problems}
For general nonconvex problems in this part, we solve the simple form problem $(\mcal{P}^g_+)$ and use
Theorem \ref{theorem7} to obtain the analytic solution $(x, v)$ to primal problem $(\mathcal{P}_b)$. Two general nonconvex examples not decoupled are tested and the corresponding results are listed below:

\begin{example}
Set $$
 A=\begin{bmatrix} 4 &  0&  1  \\
0 &  -4&  -6  \\
1 &  -6&  4 \end{bmatrix}~~\wand~~
 B=\begin{bmatrix} 7 &  -3&  -4\\
-3 &  8&  2\\
-4 &  2&  10  \end{bmatrix},
$$  $\alpha=8$, $f=(3,2,3)$ and $c=(10,6,7)$.
\end{example}

\begin{example} Set $$
 A=\begin{bmatrix} 15 &  3&  -3&  -2&  -4\\
3 &  21&  -5&  0& 2\\
-3 &  -5&  12&  0&  2\\
-2 &  0&  0&  14&  3\\
-4 &  -2&  2&  3&  6 \end{bmatrix}~~\wand~~
 B=\begin{bmatrix} 13 &  2&  -4&  4&  -6\\
2 &  6&  -4&  1& -2\\
-4 &  -4&  6&  0&  -3\\
4 &  1&  0&  7&  -7\\
-6 &  -2&  -3&  -7&  21 \end{bmatrix},
$$  $\alpha=4$, $f=(6,-2,5,4,10)$ and $c=(7,-3,10,-4,-3)$.
\end{example}

The corresponding results are listed below:

\begin{table}[H]
\centering
\scalebox{0.8}[0.8]{%
\begin{tabular}{|c|l|c|c|l|c|}
\hline
Experiments &Dual solution $(\bar\varsigma,\bar\sigma_1)$ &$\vP^g(\bar\varsigma,\bar\sigma_1)$ &$\lambda_{min}(\bar\varsigma,\bar\sigma_1)$ & Primal solution $(\bar x, \bar v)$& $P(\bar x, \bar v)$ \\
\hline
\multirow{2}{*}{Example 7}&$\bar\varsigma=-0.5$, &\multirow{2}{*}{-33.875} & \multirow{2}{*}{1.58694}& $\bar x=(1,1,1)$,& \multirow{2}{*}{-33.875} \\
& $\bar\sigma_1=(2.5,9.75,6)$.& & & $\bar v=(1,1,1)$. & \\
\hline
\multirow{2}{*}{Example 8}&$\bar\varsigma=0.088$, &\multirow{2}{*}{-32.8777} & \multirow{2}{*}{5.54327}& $\bar x=(0.556,0,0.978,-0.174,-0.225)$,& \multirow{2}{*}{-32.8777} \\
& $\bar\sigma_1=(0,1.994,0,0,0)$.& & & $\bar v=(1,0,1,1,1)$. & \\
\hline
\end{tabular}}
\caption{Results by solving problem $(\mcal{P}^g_+)$ in Examples 7 and 8.}
\end{table}

From Table 7, it holds that  $\bar\varsigma\geq -\alpha,\bar\sigma_1\geq 0, G(\bar\varsigma,\bar\sigma_1)\succ 0, f_i+{(\bar\sigma_1)}_i\neq 0, \forall i=1,\cdots,n$ in every example, so
$(\bar \varsigma,\bar \sigma_1)\in S^+_{\varsigma\sigma_1}$. By Theorem \ref{theorem7}, the solution of primal problem is obtained. So our method is also effective for the general problems.

\section{Conclusions and further work}\label{se:concl}
In this paper we propose a canonical duality method for solving a mixed-integer nonconvex fourth-order polynomial minimization problem
with fixed cost terms. By rewriting the box constraints in a relaxed quadratic form, a relaxed reformulation of the primal
problem $(\mathcal{P}_b)$ is obtained, then the canonical dual problem $(\mcal{P}^\sharp)$ is defined and the complementary-dual principle is proved.
The primal problem $(\mathcal{P}_b)$ is canonically dual to a concave maximization
problem $(\mcal{P}^\sharp_+)$ where a useful space $S^+_\sharp$ is introduced. This constrained nonconvex problem $(\mathcal{P}_b)$ in $\mathbb{R}^{2n}$ can be transformed into a continuous concave maximization dual problem $(\mcal{P}^g_+)$ in $\mathbb{R}^{n+1}$  without duality gap. The global optimality conditions are proposed and the existence and uniqueness criteria are illustrated. Application to the decoupled mixed-integer
problem is illustrated and analytic solution for a global minimum is obtained under some suitable conditions. Several examples are given to show our method is effective. Canonical duality theory is a potentially powerful methodology, which can be used to model complex
systems to a wide class of discrete and continuous problems in global optimization
and nonconvex analysis. The ideas and results with canonical duality theory presented in this paper can be used or generalized for solving other type of problems in the future. \\

\end{document}